\documentclass[11pt]{amsart}
\usepackage{amsfonts,amsmath,amsthm}
\usepackage{mathrsfs}
\usepackage{graphicx,color}
 \newtheorem{thm}{Theorem}[section]
 
 \newtheorem{lem}[thm]{Lemma}{\rm}
 \newtheorem{prop}[thm]{Proposition}

 {\rm}
 
 \newtheorem{rem}[thm]{Remark}
 \newtheorem{ex}{Example}\numberwithin{equation}{section}

\def\x{\boldsymbol{x}}
\def\y{\boldsymbol{y}}

\def\u{\boldsymbol{u}}

\def\A{\mathbf{A}}
\def\B{\mathbf{B}}

\def\bxi{\boldsymbol{\xi}}
\def\bzeta{\boldsymbol{\zeta}}

\def\v{\boldsymbol{v}}

\def\p{\boldsymbol{p}}

\def\K{\mathbf{K}}

\def\M{\mathbf{M}}

\def\N{\mathbb{N}}
\def\R{\mathbb{R}}
\def\M{\mathbf{M}}
\def\om{\mathbf{\Omega}}

\def\bxi{\boldsymbol{\xi}}
\def\bmu{\boldsymbol{\mu}}

\def\balpha{\boldsymbol{\alpha}}
\def\bbeta{\boldsymbol{\beta}}

\def\taue{\tau_{\varepsilon}}

\def\tauer{\tau_{\varepsilon/r}}
\begin{document}
\title[A regularized Christoffel function]{A modified Christoffel function and its asymptotic properties}
\author{Jean B. Lasserre}
\address{LAAS-CNRS and Institute of Mathematics\\
University of Toulouse\\
LAAS, 7 avenue du Colonel Roche\\
31077 Toulouse C\'edex 4, France\\
Tel: +33561336415}
\email{lasserre@laas.fr}

\date{}

\thanks{This work is supported by the AI Interdisciplinary Institute ANITI funding through the french program ``Investing for the Future PI3A" under the grant agreement number ANR-19-PI3A-0004. This research is also part of the programme DesCartes and is supported by the National Research Foundation, Prime Minister's Office, Singapore under its Campus for Research Excellence and Technological Enterprise (CREATE) programme.}
\begin{abstract}
 We introduce a certain variant (or regularization) $\tilde{\Lambda}^\mu_n$ of the standard 
 Christoffel function $\Lambda^\mu_n$ associated with a measure $\mu$ on a compact set $\om\subset\R^d$. 
  Its reciprocal is now a sum-of-squares polynomial
 in the variables $(\x,\varepsilon)$, $\varepsilon>0$. It shares the same dichotomy
 property of the standard Christoffel function, that is, the growth with $n$ of its inverse is at most polynomial inside  and exponential outside the support of the measure.
 Its distinguishing and crucial feature states that for fixed $\varepsilon>0$, and under weak assumptions, 
 $\lim_{n\to\infty} \varepsilon^{-d}\tilde{\Lambda}^\mu_n(\bxi,\varepsilon)=f(\bzeta_\varepsilon)$ where $f$ (assumed to be continuous) is the unknown 
 density of $\mu$ w.r.t. Lebesgue measure on $\om$, and
 $\bzeta_\varepsilon\in\B_\infty(\bxi,\varepsilon)$ (and so $f(\bzeta_\varepsilon)\approx f(\bxi)$ when $\varepsilon>0$ is small).
 This is in contrast
 with the standard Christoffel function where if $\lim_{n\to\infty} n^d\Lambda^\mu_n(\bxi)$
 exists, it is of the form  $f(\bxi)/\omega_E(\bxi)$ 
 where $\omega_E$ is the density of the equilibrium measure of $\om$, usually unknown.
  At last but not least, the additional computational burden (when compared to computing $\Lambda^\mu_n$) is just integrating symbolically the monomial basis $(\x^{\balpha})_{\balpha\in\N^d_n}$
 on the box $\{\x: \Vert \x-\bxi\Vert_\infty<\varepsilon/2\}$, so that $1/\tilde{\Lambda}^\mu_n$
 is obtained as an explicit polynomial of $(\bxi,\varepsilon)$.
\end{abstract}
\maketitle

\section{Introduction}

The Christoffel function $\Lambda^\mu_n:\R^d\to\R_+$ and the Christoffel-Darboux kernel are both associated with a measure $\mu$ whose support $\om\subset\R^d$
is here a compact set, and are indexed by a ``degree" $n$. They originate in the theory of approximation and orthogonal polynomials. Among its interesting features:

$\bullet$  (1) The Christoffel function identifies the support of $\mu$ by a dichotomy of its behavior as $n$ grows, depending on whether the point 
is inside or outside the support of the measure. Namely, inside the support the growth of its reciprocal  is at most polynomial in the degree $n$, and at least exponential outside the support.

$\bullet$ (2) Under some regularity assumptions,  the Christoffel function is also related
to the density $f$ of $\mu$ w.r.t. Lebesgue measure on $\om$. Indeed then ${d+n\choose d}\Lambda^\mu_n\to f/\omega_E$ pointwise in
$\om$ (and uniformly on compact subsets of $\om$) where $\omega_E$ is a so-called \emph{equilibrium measure} of $\om$. 
However in general the equilibrium measure is unknown, except in a few special cases of sets $\om$ with special geometry; see e.g. \cite[\S 9.7 ]{dunkl-xu}, \cite{Kroo-a,Kroo}, \cite{Xu}.

$\bullet$ (3) Its explicit expression as a rational function, is easy to compute, e.g. via
the inverse of the moment matrix of $\mu$ (provided that its size is reasonable and compatible with state-of-the-art linear algebra softwares).

As advocated in \cite{lass-book}, in addition to being a mathematical object  interesting in its own, the Christoffel function (CF)  turns out 
to also provide an efficient and easy-to-use tool to help solve some problems in data analysis, e.g. outlier detection,
support inference, and density estimation. In such problems, the  underlying measure of interest $\mu$ is usually not available and one has rather access to a \emph{finite} sample of data points generated according to $\mu$. 
Then in lieu of the unknown $\mu$, one uses the empirical measure associated with the cloud of finitely many data points.
Remarkably, and even though the geometry of the support of the empirical measure is quite trivial, 
its associated CF is still close (in a certain sense) to that of $\mu$, and inherits interesting features of $\mu$ that can be exploited in data analysis; see e.g. 
\cite{nips,cras-classif,lass-book,advc}.

Moreover, some recent works 
have revealed additional properties of the CF, as well as connections (in the author's opinion some even surprising) with
seemingless disconnected topics, e.g., convex duality, certificates of positivity in real algebraic geometry, Pell's equation, equilibrium measure of a compact set; see \cite{cras-disintegration,pell}.

For more details on the Christoffel-Darboux kernel and the Christoffel function, the interested reader is referred to e.g. 
\cite{dunkl-xu,lass-book,Nevai,Simon,Totik,Xu} and  many references therein.\\

\subsection*{Contribution}~
Our motivation is essentially concerned with point (2) above.
Indeed our main goal is to still recover the density $f$ of $\mu$ asymptotically, but without
the annoying factor $\omega_E$ of the equilibrium measure because it is unknown 
in general. But simultaneously, we also want to maintain the support inference capability in (1) as well as 
an efficient computation (point  (3) above).

To achieve these goals, we introduce the following variant (or regularization) of the Christoffel function, namely the function 
$\tilde{\Lambda}^\mu_n:\R^d\times\R\to\R_+$ defined by:
\begin{equation}
\label{intro-1}
 \tilde{\Lambda}^\mu_n(\bxi,\varepsilon)
 \,:=\,\,\{ \inf_{p\in\R[\x]_n}\,\{\, \int p^2\,d\mu:\quad \int_{\B_{\infty}(\bxi,\varepsilon)}p(\x)\,\frac{d\x}{\varepsilon^d}\,=\,1\,\}\,,
 \end{equation}
for all $(\bxi,\varepsilon)\in\R^d\times\R_+$ and $n\in\N$, where 
$\B_{\infty}(\bxi,\varepsilon):=\{\x:\Vert\x-\bxi\Vert_\infty<\varepsilon/2\}$ (with Lebesgue volume $\varepsilon^d$).

As we will see, $1/\tilde{\Lambda}^\mu_n$ is obtained as an \emph{explicit} polynomial of
$(\bxi,\varepsilon)$, and can be also seen as a
a polynomial of $\bxi$, parametrized by $\varepsilon>0$ fixed.
In particular, we are interested in the asymptotic behaviors of
$\tilde{\Lambda}^\mu_n(\bxi,\varepsilon)$ with $\varepsilon>0$ fixed, as well as
$\tilde{\Lambda}^\mu_n(\bxi,1/n)$, as $n$ grows.\\

More precisely, let $\om\subset\R^d$ be compact with nonempty interior, and let $\mu$
be a Borel measure on $\om$ with density $f$ with respect to (w.r.t.) Lebesgue measure on $\om$.
In particular, for every $n\in\N$, the moment matrix $\M_n(\mu)$ is non singular.
We prove that $\tilde{\Lambda}^\mu_n$ in \eqref{intro-1} has the following properties:

$\bullet$ $1/\tilde{\Lambda}^\mu_n$ is a sum of squares polynomial with explicit form:
\begin{equation}
\label{intro-2}
(\bxi,\varepsilon)\mapsto \tilde{\Lambda}^\mu_n(\bxi,\varepsilon)^{-1}\,=\,\tilde{\v}_n(\bxi,\varepsilon)^T\M_n(\mu)^{-1}
\tilde{\v}_n(\bxi,\varepsilon)\,,\end{equation}
where, with $\v_n(\x)$ being the vector of the monomial basis $(\x^{\balpha})_{\balpha\in\N^d_n}$,
\[\tilde{\v}_n(\bxi,\varepsilon)\,:=\,
\int_{\B_{\infty}(\bxi,\varepsilon)}\v_n(\x)\,\frac{d\x}{\varepsilon^d}\,.\]
Importantly, each entry $\tilde{\v}_{n,\balpha}$, $\balpha\in\N^d_n$, of the vector $\tilde{\v}_n$, is obtained explicitly as a 
\emph{polynomial} in $(\bxi,\varepsilon)$.

$\bullet$ With $\varepsilon>0$ fixed, and assuming that $f$ is bounded and $f>0$ on $\om$, 
\begin{equation}
\label{intro-3}
\lim_{n\to\infty}\varepsilon^d\,\tilde{\Lambda}^\mu_n(\bxi,\varepsilon)^{-1}\,=\,
\int_{\B_{\infty}(\bxi,\varepsilon)}\frac{1}{f}\,\frac{d\x}{\varepsilon^d}\,=\,\varepsilon^d\,
\Vert \frac{1_{\B_\infty(\bxi,\varepsilon)}}{\varepsilon^d\,f}\Vert^2_{L^2(\mu)}\,,\end{equation}
for all $\bxi$ with $\B_{\infty}(\bxi,\varepsilon)\subset\om$.  Moreover, if $f$ is continuous then
\begin{equation}
\label{intro-33}
\lim_{n\to\infty}\varepsilon^{-d}\,\tilde{\Lambda}^\mu_n(\bxi,\varepsilon)\,=\,
f(\bzeta_\varepsilon)\,,
\end{equation}
for some $\bzeta_\varepsilon\in\B_\infty(\bxi,\varepsilon)$, and so $f(\bzeta_\varepsilon)\approx f(\bxi)$ if $\varepsilon$ is small enough.

$\bullet$ Finally, and as for the standard CF $\Lambda^\mu_n$,
the growth with $n$ of $\tilde{\Lambda}^\mu_n(\bxi,\varepsilon)^{-1}$ ($\varepsilon>0$ fixed) and
$\tilde{\Lambda}^\mu_n(\bxi,1/n)^{-1}$,
is at most polynomial if $\bxi\in\mathrm{int}(\om)$  and at least exponential
if $\bxi\not\in\om$.\\ 

So the extended CF $\tilde{\Lambda}^\mu_n$ essentially shares same important features as the standard  CF with the advantage 
that the limit of $\varepsilon^{-d}\tilde{\Lambda}^\mu_n(\bxi,\varepsilon)$ ($\varepsilon>0$ fixed) as $n$ increases, is 
close to $f(\bxi)$ where $f$ is the density of $\mu$, a highly desirable feature. 
Moreover, the 	additional price to pay to obtain its explicit form \eqref{intro-2} as a function of $(\bxi,\varepsilon)$, is rather negligible 
as it only requires the closed form expression of $\int_{\B_\infty(\bxi,\varepsilon)}\x^{\balpha}d\x$,
for all $\balpha\in\N^d_n$.

\subsection*{Interpretation} ~For fixed $\varepsilon>0$, 
the function $\bxi\mapsto \tilde{\Lambda}^\mu_n(\bxi,\varepsilon)$ 
has a simple interpretation.  Let $f$ be the density of $\mu$ with respect to Lebesgue measure on $\om$ and assume
that $f\geq\gamma>0$ on $\om$.
Let $\ell^\varepsilon_{\bxi}$ be the linear functional $h\mapsto \ell^\varepsilon_{\bxi}(h):=\int  h\,d\lambda$ where $\lambda$ is the uniform probability mesure on 
$\B_\infty(\bxi,\varepsilon)$. If $\bxi\in\mathrm{int}(\om)$ is such that $\B_\infty(\bxi,\varepsilon)\subset\om$, then
$\ell^\varepsilon_{\bxi}\in L^2(\mu)$ and  is the function 
$\x\mapsto 1_{\B_\infty(\bxi,\varepsilon)}(\x)/\taue\,f\in L^2(\mu)$ where $\taue=\mathrm{vol}(\B_\infty(\bxi,\varepsilon))$.
Next, let $\ell^\varepsilon_{\bxi,n}\in\R[\x]_n$ be the orthogonal projection of $\ell^\varepsilon_{\bxi}$
on the finite-dimensional subspace $\R[\x]_n$ of  $L^2(\mu)$.
Then $\tilde{\Lambda}^\mu_n(\bxi,\varepsilon)=1/\Vert\ell^\varepsilon_{\bxi,n}\Vert^2_{L^2(\mu)}$ with limit \eqref{intro-3} as $n\to\infty$. 

$\tilde{\Lambda}^\mu_n(\bxi,\varepsilon)$
can be seen as an $\varepsilon$-regularization of $\Lambda^\mu_n(\bxi)$, where instead
of working with the point evaluation at $\bxi$ (i.e., the Dirac measure $\delta_{\bxi}$ 
which is not in $L^2(\mu)$), one rather works with
the element $\ell^\varepsilon_{\bxi}$ of some neighborhood $\mathcal{N}(\delta_{\bxi},\eta)$ of $\delta_{\bxi}$ (in the weak-$\star$ topology of  the space of finite Borel signed measures)
because it can also be considered as an element of $L^2(\mu)$. In doing so with $\varepsilon>0$ fixed, $\lim_{n\to\infty}\varepsilon^{-d}\tilde{\Lambda}^\mu_n(\bxi,\varepsilon)$ exists (which is not the case for $\Lambda^\mu_n(\bxi)$) and is close to $f(\bxi)$ if $\varepsilon$ is small and $f$ is continuous, a highly desirable feature.

\section{Main result}

\subsection{Notation and definitions}

Let $\R[\x]$ denote the ring of real polynomials in the variables $\x=(x_1,\ldots,x_d)$ and $\R[\x]_n\subset\R[\x]$ be its subset 
of polynomials of total degree at most $n$. 
Let $\N^d_n:=\{\balpha\in\N^d:\vert\balpha\vert\leq n\}$
(where $\vert\balpha\vert=\sum_i\alpha_i$) with cardinal 
$s(n)={d+n\choose d}$. Let $\v_n(\x)=(\x^{\balpha})_{\balpha\in\N^d_n}$ 
be the vector of monomials up to degree $n$, 
and let $\Sigma[\x]_n\subset\R[\x]_{2n}$ be the convex cone of polynomials of total degree at most $2n$ which are sum-of-squares (in short SOS). For every $p\in\R[\x]_n$ write
\[\x\mapsto p(\x)\,=\,\langle\p,\v_n(\x)\rangle\,,\quad\forall \x\in\R^d\,,\]
where $\p\in\R^{s(n)}$ is the vector of coefficients of $p$ in the monomial basis $(\x^{\balpha})_{\balpha\in\N^d}$.
For a real symmetric matrix 
$\A=\A^T$ the notation $\A\succeq0$ (resp. $\A\succ0$) stands for $\A$ is positive semidefinite (p.s.d.) (resp. positive definite (p.d.)).

The support of a Borel measure $\mu$ on $\R^d$ is the smallest closed set $A$ such that
$\mu(\R^d\setminus A)=0$, and such a  set $A$ is unique. With $S\subset\R^d$ compact, denote by $\mathscr{C}(S)$ 
the Banach space of real continuous functions on $S$ equipped with the sup-norm. Its topological dual
$\mathscr{C}(S)^*$ is the Banach space $\mathscr{M}(S)$
of finite signed Borel measures on $S$, 
equipped with the total-variation norm.

\subsection*{Moment matrix.} 

Let $\mu$ be a finite Borel measure on $\R^d$ with all moments $\bmu=(\mu_{\balpha})_{\balpha\in\N^d}$
assumed to be finite.
The (degree-$n$) moment matrix $\M_n(\mu)$ associated with $\mu$ is the real symmetric matrix 
with rows and columns indexed by $\N^d_n$ (hence of size $s(n)$), and with entries
\[\M_n(\mu)(\balpha,\bbeta)\,:=\,\int \x^{\balpha+\bbeta}\,d\mu\,=\,\mu_{\balpha+\bbeta}\,,\quad\balpha,\bbeta\in\N^d_n\,.\]
Obviously, $\M_n(\mu)\succeq0$ for all $n$ since 
\[\langle \p,\M_n(\mu)\,\p\rangle\,=\,\int p^2\,d\mu\,\geq\,0\,,\quad\forall p\in\R[\x]_n\,.\]

\subsection*{Christoffel function.} 

We here assume that $\M_n(\mu)\succ0$ for all $n\in\N$, and therefore the inverse
$\M_n(\mu)^{-1}$ is well-defined for all $n\in\N$. In particular this is true 
in our case of interest, i.e.,  when the support $\om\subset\R^d$ of $\mu$ is compact with nonempty interior and $\mu$ has 
a density w.r.t. Lebesgue measure on $\om$.

The (degree-$n$) Christoffel function $\Lambda^\mu_n:\R\to\R_+$, associated with $\mu$, is defined by:
\[\x\mapsto \Lambda^{\mu}_n(\x)^{-1}\,:=\,\v_n(\x)^T\M_n(\mu)^{-1}\v_n(\x)\,,\quad\forall \x\in\R^d\,.\]
Alternatively, if $(P_{\balpha})_{\balpha\in\N^d}\subset\R[\x]$ is a family of polynomials which are orthonormal with 
respect to $\mu$, then
\begin{equation}
 \label{ortho-poly-0}
\Lambda^{\mu}_n(\x)^{-1}\,=\,\sum_{\balpha\in \N^d_n}P_{\balpha}(\x)^2\,,\quad\forall\x\in\R^d\,.
\end{equation}
The Christoffel function has also a variational formulation. Namely:
\begin{equation}
 \label{eq:variational}
 \Lambda^\mu_n(\bxi)\,=\,\inf_{p\in\R[\x]_n}\,\{\,\int p^2\,d\mu:\quad p(\bxi)\,=\,1\,\}\,,\quad\forall\bxi\in\R^d\,.
\end{equation}
Problem \eqref{eq:variational} is a quadratic convex optimization problem that can be solved efficiently. Its
unique optimal solution $p^*\in\R[\x]_n$ reads:
\[\x\mapsto p^*(\x)\,:=\,\frac{\sum_{\balpha\in\N^d_n}P_{\balpha}(\bxi)P_{\balpha}(\x)}
{\sum_{\balpha\in\N^d_n}P_{\balpha}(\bxi)^2}\,=\,\frac{K^\mu_n(\bxi,\x)}{K^\mu_n(\bxi,\bxi)}\,,\quad\x\in\R^d\,,\]
where $K^\mu_n:\R^d\times\R^d\to\R$, defined by:
\begin{equation}
\label{def-CD-kernel}
(\x,\y)\mapsto K^\mu_n(\x,\y)\,:=\,\sum_{\balpha\in\N^n_t}P_{\balpha}(\x)\,P_{\balpha}(\y)\,,\quad \x,\y\in\R^d\,,
\end{equation}
is the Christoffel-Darboux kernel associated with $\mu$. In particular
\[\Lambda^\mu_n(\bxi)^{-1}\,=\,K^\mu_n(\bxi,\bxi)\,,\quad\forall \bxi\in\R^d\,.\]

\subsection{A regularization and parametrization of the Christoffel function}

Let $\mu$ be a finite Borel probability measure on a compact set $\om\subset\R^d$, with density $f$
w.r.t. Lebesgue measure on $\om$, 
i.e., $d\mu=1_\om(\x)f(\x)d\x$ and 
$\int_\om f(\x)\,d\x=1$. Let $L^2(\om,\mu)$ (in short, $L^2(\mu)$) be the usual Hilbert space
of square integrable functions w.r.t. $\mu$.
Next, given $\varepsilon>0$ and $\bxi\in\R^d$, let 

\begin{eqnarray}
\label{def-Binf}
\B_{\infty}(\bxi;\varepsilon)&:=&\{\x\in\R^d: \Vert \x-\bxi\Vert_\infty\leq\varepsilon/2\}\\
\label{def-tau-inf}
\tau_{\varepsilon}&:=&\mathrm{vol}(\B_{\infty}(\bxi,\varepsilon))\,=\,\varepsilon^d\\
\label{def-1}
 d\phi_{\bxi}^{\varepsilon}&=&\frac{1_{\B_{\infty}(\bxi,\varepsilon)}(\x)\,d\x}{\tau_\varepsilon}\,,\quad\varepsilon>0\,;\quad\phi_{\bxi}^0\,:=\,\delta_{\{\bxi\}}\,.
\end{eqnarray}
\begin{prop}
\label{prop1}
With $\phi_{\bxi}^{\varepsilon}$ as \eqref{def-1},
$\phi_{\bxi}^{\varepsilon}\Rightarrow\phi^0_{\bxi}\,=\,\delta_{\{\bxi\}}$ as $\varepsilon\downarrow 0$, i.e.,
\[\lim_{\varepsilon\to 0}\int h\,d\phi^{\varepsilon}_{\bxi}\,=\,h(\bxi)\,=\,\int h\,d\phi^0_{\bxi}\,,\quad\forall h\in\mathscr{C}(\R^n)\,.\]
\end{prop}
\begin{proof}
Observe that as $h$ is continuous, 
\[\int h\,d\phi_{\bxi}^\varepsilon\,=\,h(\bzeta_{\varepsilon})\,,\quad\mbox{for some $\bzeta_{\varepsilon}\in\B_{\infty}(\bxi,\varepsilon)$,}\]
and therefore $h(\bzeta_{\varepsilon})\to h(\bxi)$ as $\varepsilon\downarrow 0$.
\end{proof}
In Proposition \ref{prop1}, the notation $\phi_{\bxi}^{\varepsilon}\Rightarrow\phi^0_{\bxi}$ is standard
and stands for the weak 
convergence of probability measures, i.e., for the weak-$\star$ topology $\sigma(\mathscr{M}(\om),\mathscr{C}(\om))$
of $\mathscr{M}(\om)$.
With $\varepsilon\geq0$, the extended Christoffel function $\tilde{\Lambda}^\mu_n:\R^d\times\R_+\to\R_+$ is defined by:
\begin{equation}
\label{new-chris}
(\bxi,\varepsilon)\mapsto \tilde{\Lambda}^{\mu}_n(\bxi,\varepsilon)
\,:=\,\displaystyle\min_{p\in\R[\x]_n}\,\{\,\int_\om p^2d\mu\,:\: \displaystyle\int p\,d\phi_{\bxi}^{\varepsilon}=1\,\}\,,\quad\bxi\in\R^d\,,
\end{equation}
where we have included the case $\varepsilon=0$  with $\phi^0_{\bxi}=\delta_{\{\bxi\}}$.
Notice that the standard Christoffel function $\Lambda^\mu_n$ satisfies
\begin{eqnarray*}
\Lambda^\mu_n(\bxi)&=&\inf_{p\in \R[\x]_n}\,\{\,\int p^2\,d\mu:\: p(\bxi)\,=\,1\,\}\\
&=&\inf_{p\in \R[\x]_n}\,\{\,\int p^2\,d\mu:\: \int p\,d\phi^0_{\bxi}\,=\,1\,\}\,=\,\tilde{\Lambda}^\mu_n(\bxi,0)\,,\quad\bxi\in\R^d\,,
\end{eqnarray*}
that is, $\tilde{\Lambda}^\mu_n(\bxi,0)=\Lambda^\mu_n(\bxi)$, for all $\bxi\in\R^d$.

\subsection*{An explicit form.}
Importantly, the extended Christoffel function $\tilde{\Lambda}^\mu_n$
is obtained as an \emph{explicit}  rational function of $\bxi$ and $\varepsilon$. More precisely,
$1/\tilde{\Lambda}^\mu_n$ is obtained as an explicit sum-of-squares (SOS) polynomial
of $(\bxi,\varepsilon)$. 
\begin{lem}
\label{lem3}
Let $\tilde{\Lambda}^\mu_n$ be as in \eqref{new-chris} and let $(P_{\balpha})_{\balpha\in\N^d_t}$ be a family of polynomials orthonormal w.r.t. $\mu$. Then the unique 
optimal solution $p^*_n\in\R[\x]_n$ of \eqref{new-chris} satisfies:
\begin{eqnarray}
\label{lem3-0}
\x\mapsto p_n^*(\x)&=&
\tilde{\Lambda}^{\mu}_{n}(\bxi,\varepsilon)\,\v_n(\x)^T\M_n(\mu)^{-1}\,\int\v_n(\y)\,d\phi^{\varepsilon}_{\bxi}(\y)\\
\label{lem3-00}&=&\tilde{\Lambda}^{\mu}_{n}(\bxi,\varepsilon)\,\int K^\mu_n(\x,\y)\,d\phi^{\varepsilon}_{\bxi}(\y)\,.
\end{eqnarray}
In addition,
\begin{equation}
\label{lem3-3}
\tilde{\Lambda}^\mu_n(\bxi,\varepsilon)^{-1}
\,=\,\tilde{\v}_n(\bxi,\varepsilon)^T\,\M_n(\mu)^{-1}\tilde{\v}_n(\bxi,\varepsilon)\,,\quad\bxi\in\R^n\,,
\end{equation}
where $\tilde{\v}_n\in\R[\x,\varepsilon]_n$ is defined by:
\begin{equation}
\label{def-tilde}
\bxi\mapsto \tilde{\v}_n(\bxi,\varepsilon)\,:=\,\frac{1}{\taue}\,\int_{\B_{\infty}(\bxi,\varepsilon)} \v_n(\y)\,d\y\,,\quad n\in\N\,.
\end{equation}
\end{lem}
\begin{proof}
 Rewrite \eqref{new-chris} 
 \[\tilde{\Lambda}^{\mu}_n(\bxi,\varepsilon)\,=\,\displaystyle\min_{\p\in\R^{s(n)}}\,\{\,\langle\p,\M_n(\mu)\,\p\rangle:\quad
 \langle\,\p\,,\int \v_n(\x)\,d\phi_{\bxi}^{\varepsilon}\,\rangle =1\,\}\,,\quad\bxi\in\R^d\,,\]
 which is a convex quadratic optimization problem. Its unique optimal solution $\p^*_n\in\R^{s(n)}$ satisfies
 \[2\,\M_n(\mu)\,\p^*_n\,=\,\lambda^*\,\int \v_n(\y)\,d\phi_{\bxi}^{\varepsilon}(\y)\,,\]
 for some scalar $\lambda^*$. Hence $\lambda^*=2\,\tilde{\Lambda}^\mu_n(\bxi,\varepsilon)$, and
 \[\p^*_n\,=\,\frac{\lambda^*}{2}\,\M_n(\mu)^{-1}\,\int \v_n(\y)\,d\phi_{\bxi}^{\varepsilon}(\y)\,,\]
 which in turn yields
 \begin{eqnarray*}
 \x\mapsto p^*_n(\x)\,=\,\langle\p^*_n,\v_n(\x)\rangle &=&\frac{\lambda^*}{2}\,\left\langle \v_n(\x),\M_n(\mu)^{-1}\,\int \v_n(\y)\,d\phi_{\bxi}^{\varepsilon}(\y)\right\rangle\\
 &=&\tilde{\Lambda}^\mu_n(\bxi,\varepsilon)\,\int K^\mu_n(\x,\y)\,d\phi_{\bxi}^{\varepsilon}(\y)\,,
 \end{eqnarray*}
 which yields  \eqref{lem3-0}-\eqref{lem3-00}. Next, using the definition \eqref{def-CD-kernel} of $\K^\mu_n$,
 and $\int P_{\balpha}\,P_{\bbeta}\,d\mu=\delta_{\balpha=\bbeta}$, for all $\balpha,\bbeta\in\N^d_n$,
  \begin{equation}
  \label{aux-50}
  \tilde{\Lambda}^\mu_n(\bxi,\varepsilon)\,=\,\int (p^*_n)^2\,d\mu\,=\,\tilde{\Lambda}^\mu_n(\bxi,\varepsilon)^2\,
  \sum_{\balpha\in\N^d_n}\left(\int_{\B_{\infty}(\bxi,\varepsilon)} P_{\balpha}(\y)\,\frac{d\y}{\taue}\right)^2\,.\end{equation}
  To obtain \eqref{lem3-3}, just use that with $(P_{\balpha}(\x))_{\balpha\in\N^d_n}=\mathbf{D}\v_n(\x)$, 
  by orthogonality of the $(P_{\balpha})$'s w.r.t. $\mu$, the resulting change of basis matrix $\mathbf{D}\in\R^{s(n)\times s(n)}$ satisfies 
  $\mathbf{D}^T\mathbf{D}=\M_n(\mu)^{-1}$.
  
  Finally it remains to prove that $\tilde{\v}_n\in\R[\bxi,\varepsilon]_n$. With $\bbeta\in\N^d_n$,
\begin{equation}
\label{expand}
\int_{\B_{\infty}(\bxi,\varepsilon)}\y^{\bbeta}\,\frac{d\y}{\taue}\,=\,\varepsilon^{-d}\prod_{i=1}^d \frac{(\xi_i+\varepsilon/2)^{\beta_i+1}-(\xi_i-\varepsilon/2)^{\beta_i+1}}{\beta_i+1}\end{equation}
is indeed a polynomial in $(\bxi,\varepsilon)$. To see this, 
in each term of the product in \eqref{expand},
use the identity $a_i^{n+1}-b_i^{n+1}=(a_i-b_i)\,(\sum_{j=0}^{n}a_i^{n-j}b_i^j)$
with $a_i=\xi_i+\varepsilon/2$ and $b_i=\xi_i-\varepsilon/2$, so that $a_i-b_i=\varepsilon$ for ever $i$, which implies that
the term $\varepsilon^{-d}$ is annihilated.
This implies that $\tilde{\v}_n(\bxi,\varepsilon)=\int_{\B_{\infty}(\bxi,\varepsilon)}\v_n(\y)\,d\y/\taue$  is a 
polynomial vector in the variables $(\bxi,\varepsilon)$.
\end{proof}
~

\subsection{Computation}~

$\bullet$ If the orthonormal polynomials $(P_{\balpha})_{\balpha\in\N^d}$ are available then just use \eqref{aux-50}
to (i)  compute the polynomials $q_{\balpha}\in\R[\bxi,\varepsilon]_n$ defined by:
\[\x\mapsto q_{\balpha}(\bxi,\varepsilon)\,:=\,\frac{1}{\taue}\int_{\B_{\infty}(\bxi,\varepsilon)} P_{\balpha}(\y)\,d\y\,,\quad \balpha\in\N^d_n\,,\]
and (ii), sum up $\sum_{\balpha\in\N^d_n}q_{\balpha}(\bxi,\varepsilon)^2$.

$\bullet$ On the other hand, if the moment matrix $\M_n(\mu)$ is available then 
(i) compute the polynomial vector $\tilde{\v}_n(\bxi,\varepsilon)\in\R[\bxi,\varepsilon]_n^{s(n)}$ in
\eqref{def-tilde} and (ii), form the SOS polynomial $\tilde{\v}_n(\bxi,\varepsilon)^T\,\M_n(\mu)^{-1}\tilde{\v}_n(\bxi,\varepsilon)/\taue^2$ to obtain
the SOS polynomial $\tilde{\Lambda}^{\mu}_{n}(\bxi,\varepsilon)^{-1}$.

In both cases, computing the polynomials $q_{\balpha}$ or $\tilde{\v}_n$ is an easy task
which can be done exactly and even symbolically in $(\bxi,\varepsilon)$, as it reduces to integrate a polynomial
on the box $\B_{\infty}(\bxi,\varepsilon)$ parametrized by $\bxi$ and $\varepsilon$. 

Therefore it is worth emphasizing that in the end, one thus obtains the polynomial $(\bxi,\varepsilon)\mapsto
\tilde{\Lambda}^\mu_n(\bxi,\varepsilon)^{-1}$ as an \emph{explicit} SOS polynomial in the variables 
$(\bxi,\varepsilon)$, via 
$\tilde{\v}_n(\bxi,\varepsilon)^T\,\M_n(\mu)^{-1}\tilde{\v}_n(\bxi,\varepsilon)$, exactly as
$\bxi\mapsto\Lambda^\mu_n(\bxi)^{-1}$ was obtained as an explicit SOS polynomial in $\bxi$, via $\v_n(\bxi)^T\,\M_n(\mu)^{-1}\v_n(\bxi)$.

 Compared with computation of the standard Christoffel function $\Lambda^\mu_n$,
computing the extended Christoffel function $\tilde{\Lambda}^\mu_n$ only requires
an extra symbolic integration  (e.g., of the vector  $\v_n(\x)=(\x^{\balpha})_{\balpha\in\N^d_n}$
on the box $\B_{\infty}(\bxi,\varepsilon)$), an easy task once (as long as the dimension $d$ is not too large).
In addition, from inspection of \eqref{expand} we may conclude that
\[\tilde{\v}_n(\bxi,\varepsilon)\,=\,\left\{\begin{array}{rl}
\v_n(\bxi)&\mbox{if $n\leq 1$,}\\
\v_n(\bxi)+O(\varepsilon)&\mbox{if $n>1$,}\end{array}\right.\]
which in turn yields:
\begin{equation}
\label{asymptotic}
\tilde{\Lambda}^\mu_n(\bxi,\varepsilon)^{-1}\,=\,\Lambda^\mu_n(\bxi)^{-1}+O(\varepsilon)\,,\quad\forall\bxi\in\R^d\,.\end{equation}
\begin{ex}
\label{ex-1}
Let $\om=[-1,1]$ and $\mu=dx/\sqrt{1-x^2}$, so that the Chebyshev polynomials
of first kind $(T_j)_{j\in\N}$ are orthogonal w.r.t. $\mu$, and the family
$(P_j)_{j\in\N}$ with $P_0=T_0/\sqrt{\pi}$ and $P_j=T_j\sqrt{2/\pi}$, $j\geq1$, is orthonormal w.r.t. $\mu$. Then $\Lambda^\mu_0(\xi)=1/\pi=\tilde{\Lambda}^\mu_0(\xi,\varepsilon)$, and 
\begin{eqnarray*}
\Lambda_1^\mu(\xi)&=&\frac{2}{\pi}\left(\,\frac{1}{2}+\xi^2\right)\,=\,\tilde{\Lambda}^\mu_1(\xi,\varepsilon)\\
\Lambda_2^\mu(\xi)&=&\frac{2}{\pi}\left(\,\frac{3}{2}-3\,\xi^2+4\,\xi^4\,\right)\\
\tilde{\Lambda}_2^\mu(\xi,\varepsilon)&=&\frac{2}{\pi}\left(\,\frac{3}{2}-3\,\xi^2+4\,\xi^4+\frac{\varepsilon^2}{2}-\frac{2\,\xi^2\varepsilon^2}{3}\right)\\
&=&\Lambda^\mu_2(\bxi)+\frac{\varepsilon^2}{\pi}\,(1-\frac{4\,\xi^2}{3})\,.
\end{eqnarray*}
\end{ex}

\subsection{A $L^2(\mu)$-norm interpretation}

Recall that $d\mu=f\,d\x$ for some unknown density $f:\om\to\R_+$, and $f\geq\gamma$ for some $\gamma>0$.
In particular, $\x\mapsto f(\x)^{-1}$ is an element of $L^2(\mu)$.

\begin{lem}
\label{lem2}
Assume that  $f\geq\gamma$ on $\om$, for some $\gamma>0$. 
Let $\bxi\in\om$ and $\varepsilon>0$ be fixed. 

(i) If $\B_\infty(\bxi,\varepsilon)\subset\om$ then the linear functional $\ell_{\bxi}^{\varepsilon}:\R[\x]\to\R$,
\[h\mapsto \ell_{\bxi}^{\varepsilon}(h):=\frac{1}{\taue}\,\int_{B(\bxi,\varepsilon)}h(\x)\,d\x\,,\quad \forall h\in \R[\x]\,,\]
is an element of $L^2(\mu)$, that we still denote $\ell_{\bxi}^{\varepsilon}$. Moreover,
it is represented by the function
\[\x\,\mapsto \frac{1_{\B_\infty(\bxi,\varepsilon)}(\x)}{\taue\,f(\x)}\in L^2(\mu)\,,\quad\forall \x\in\om\,,\]
 hence with norm
\begin{equation}
\label{lem-norm}
\Vert \ell_{\bxi}^{\varepsilon}\Vert_{L^2(\mu)}\,=\,\left(\int_{\om} \left(\frac{1_{\B_\infty(\bxi,\varepsilon)}(\x)}{\taue\,f(\x)}\right)^2\,f\,d\x\right)^{1/2}\,=\,\left\Vert \frac{1_{\B_\infty(\bxi,\varepsilon)}}{\taue\,f}\right\Vert_{L^2(\mu)}\,.
\end{equation}
(ii) In addition, if $f$ is continuous  then
\begin{equation}
\label{lem2-1}
\lim_{\varepsilon\to 0}\taue\,\Vert\ell^{\varepsilon}_{\bxi}\Vert^2_{L^2(\mu)}\,=\,\frac{1}{f(\bxi)}\,,
\quad\forall\bxi\in\mathrm{int}(\om)\,.
\end{equation}
\end{lem}
\begin{proof}
(i) The linear functional $\ell_{\bxi}^\varepsilon$ is a Borel measure on $\B_{\infty}(\bxi,\varepsilon)$ with density $1/\taue$ w.r.t. Lebesgue measure on $\B_{\infty}(\bxi,\varepsilon)$. If 
$\B_{\infty}(\bxi,\varepsilon)\subset\om$ and $f\geq\gamma>0$ on $\om$, then 
\begin{eqnarray*}
\ell_{\bxi}^{\varepsilon}(h)\,=\,\frac{1}{\taue}\int_{\B_\infty(\bxi,\varepsilon)}h(\x)\,d\x&=&
\int_{\om}h(\x)\,\frac{1_{\B_\infty(\bxi,\varepsilon)}(\x)}{\taue\,f(\x)}\,f(\x)\,d\x\\
&=&\int_{\om}h(\x)\,\frac{1_{\B_\infty(\bxi,\varepsilon)}(\x)}{\taue\,f(\x)}\,d\mu\\
&=&\left\langle h,\frac{1_{\B_\infty(\bxi,\varepsilon)}}{\taue\,f}\right\rangle\,,\quad\forall h\in L^2(\mu)\,,
\end{eqnarray*}
and so the linear functional 
$\ell^\varepsilon_{\bxi}$ is represented by $\frac{1_{\B_\infty(\bxi,\varepsilon)}}{\taue\,f}\in L^2(\mu)$.
It is simply the measure on $\om$ with density $1_{\B_{\infty}(\bxi,\varepsilon)}/\taue f$ w.r.t. $\mu$, hence an element of $L^2(\mu)$ with norm $\Vert 1_{\B_{\infty}(\bxi,\varepsilon)}/\taue f\Vert_{L^2(\mu)}$.

(ii) In addition,  if $f$ is continuous then
\begin{eqnarray*}
\Vert \ell_{\bxi}^{\varepsilon}\Vert_{L^2}^2\,=\,\int_{\om} \left(\frac{1_{\B_\infty(\bxi,\varepsilon)}(\x)}{\taue\,f(\x)}\right)^2\,d\mu
&=&\int_{\om} \left(\frac{1_{\B_\infty(\bxi,\varepsilon)}(\x)}{\taue\,f(\x)}\right)^2\,f\,d\x\\
&=&\int_{\B_\infty(\bxi,\varepsilon)} \frac{1}{\taue\,f(\x)}\,\frac{d\x}{\taue}\,=\,\frac{1}{\taue\,f(\bzeta_{\varepsilon})}\,,
\end{eqnarray*}
for some $\bzeta_{\varepsilon}\in \B_{\infty}(\bxi,\varepsilon)$, which yields \eqref{lem2-1} when 
$\varepsilon\downarrow 0$ (as $f$ is continuous).
\end{proof}

\subsection{Relating $\ell^{\varepsilon}_{\bxi}$ with  $\tilde{\Lambda}^\mu_n$, and asymptotic properties}
~

We first start with  relating $\tilde{\Lambda}^\mu_n$ with $\Lambda^\mu_n$ for a fixed degree
``$n$".
\begin{prop}
\label{prop1}
Let $\varepsilon>0$ fixed, and assume that $f$ is continuous. Then
for every $n\in\N$:
  \begin{equation}
 \label{prop1-1}
  \tilde{\Lambda}^\mu_n(\bxi,\varepsilon)\,\geq\, \Lambda^\mu_n(\bzeta_{n})\,
  \quad \mbox{for some $\bzeta_n\in\B_\infty(\bxi,\varepsilon)$.}
 \end{equation}
In particular, using $\varepsilon:=1/n$ yields that for every $n\in\N$,
 \begin{equation}
 \label{prop1-2}
  \tilde{\Lambda}^\mu_n(\bxi,1/n)\,\geq\, \Lambda^\mu_n(\bzeta_{n})\,
  \quad \mbox{for some $\bzeta_n\in\B_\infty(\bxi,1/n)$.}
 \end{equation}
\end{prop}
\begin{proof}
With $p^*_n\in\R[\x]_n$  as in Lemma \ref{lem3},
\[1\,=\,\int_{\B_\infty(\bxi,\varepsilon)}p^*_n\,\frac{d\x}{\varepsilon^d}\,=\,p^*_n (\bzeta_n)\,\quad \mbox{for some $\bzeta_n\in\B_\infty(\bxi,\varepsilon)$,}\]
and so as $p^*_n(\bzeta_n)=1$,
\[\tilde{\Lambda}^\mu_n(\bxi,\varepsilon)\,=\,\int (p^*_n)^2\,d\mu\,\geq\, \inf_{p\in\R[\x]_n}\{\int p^2\,d\mu\,:p(\bzeta_n)=1\}\,=\, 
\Lambda^\mu_n(\bzeta_n)\,,\]
which yields \eqref{prop1-1} and \eqref{prop1-2}.
\end{proof}
We next consider the asymptotic limit of $\tilde{\Lambda}^\mu_n(\bxi,\varepsilon)$ as 
$n$ increases, when $\bxi\in\mathrm{int}(\om)$ and $\varepsilon>0$ is fixed, as well as when $\epsilon:=1/n^r$ with $r>0$.
\begin{thm}
\label{th1}
Assume that  $f\geq\gamma$ on $\om$, for some $\gamma>0$, and let $\tilde{\Lambda}^\mu_n$ be as in \eqref{new-chris}. 
Then : 

(i) With $\varepsilon>0$ fixed and for all $\bxi\in\om$ such that  $\B_{\infty}(\bxi,\varepsilon)\subset\om$: 
\begin{equation}
\label{th1-1}
\lim_{n\to\infty}\tilde{\Lambda}^{\mu}_n(\bxi,\varepsilon)^{-1}\,=\,\Vert \ell^{\varepsilon}_{\bxi}\Vert^2_{L^2}\,=\,\int_{\B_{\infty}(\bxi,\varepsilon)}\frac{1}{\taue\,f}\,\frac{d\x}{\taue}\,.
\end{equation}
Equivalently,
\begin{equation}
\label{th1-2}
\lim_{n\to\infty}\varepsilon^d\,\tilde{\Lambda}^{\mu}_n(\bxi,\varepsilon)^{-1}
\,=\,\int_{\B_{\infty}(\bxi,\varepsilon)}\frac{1}{f}\,\frac{d\x}{\taue}\,=\,\Vert 1/f\Vert_{L^1(\B_\infty(\bxi,\varepsilon),d\x/\taue)}\,,
\end{equation}
and in addition, if $f$ is continuous on $\B_\infty(\bxi,\varepsilon)$ then 
\begin{equation}
\label{th1-3}
\lim_{n\to\infty}\varepsilon^{-d}\,\tilde{\Lambda}^{\mu}_n(\bxi,\varepsilon)\,=\,f(\bzeta_{\varepsilon})\,,
\end{equation}
for some $\bzeta_{\varepsilon}\in\B_{\infty}(\bxi,\varepsilon)$.  In particular, $\lim_{n\to\infty}\varepsilon^{-d}\tilde{\Lambda}^{\mu}_n(\bxi,\varepsilon)\approx f(\bxi)$ when $\varepsilon$ is small.

(ii) For all $\bxi\in\mathrm{int}(\om)$ and all $n\in\N$, 
\[n^d\,\tilde{\Lambda}^{\mu}_n(\bxi,1/n)\,\geq \,
\inf_{\x}\{f(\x): \x\in\B_\infty(\bxi,1/n)\}\,,\]
so that if $f$ is continuous on $\B_\infty(\bxi,\varepsilon)$ then,
\begin{equation}
\label{th1-33}
\liminf_{n\to\infty}n^d\,\tilde{\Lambda}^{\mu}_n(\bxi,1/n)\geq f(\bxi)\,.
\end{equation} 
Moreover, if $r<1$ then
\begin{equation}
\label{th1-4}
\lim_{n\to\infty}n^d\,\tilde{\Lambda}^{\mu}_n(\bxi,1/n^r)\,=\,+\infty.
\end{equation}
\end{thm}
\begin{proof}
(i) By Lemma \ref{lem2}, $\ell^\varepsilon_{\bxi}\in L^2(\mu)$.
To prove \eqref{th1-1}, notice that
\[\Vert\ell_{\bxi}^{\varepsilon}\Vert^2_{L^2}\,=\,\sup_{h\in L^2(\mu)}\,\frac{\ell_{\bxi}^{\varepsilon}(h)^2}{\int h^2\,d\mu}\,,\]
and therefore,
\begin{eqnarray}
\nonumber
\frac{1}{\Vert\ell_{\bxi}^{\varepsilon}\Vert^2_{L^2}}\,=\,\inf_{h\in L^2(\mu)}\,\frac{\int h^2\,d\mu}{\ell_{\bxi}^{\varepsilon}(h)^2}
\nonumber
&=&\inf_{h\in L^2(\mu)}\,\{\int h^2\,d\mu:\: \ell_{\bxi}^{\varepsilon}(h)^2\,=\,1\}\\
\nonumber
&=&\inf_{h\in L^2(\mu)}\,\{\int h^2\,d\mu:\: \ell_{\bxi}^{\varepsilon}(h)\,=\,1\}\\
\nonumber
&\leq&\inf_{h\in \R[\x]_n}\,\{\int h^2\,d\mu:\: \ell_{\bxi}^{\varepsilon}(h)\,=\,1\}\,,\quad\forall n\in\N\\
\label{aux-dense}
&\leq&
\tilde{\Lambda}^\mu_n(\bxi,\varepsilon)\,,\quad\forall n\in\N\,,
\end{eqnarray}
and therefore, with $\varepsilon>0$ fixed, 
\begin{equation}
\label{eq:inf}
\liminf_{n\to\infty}\tilde{\Lambda}^{\mu}_n(\bxi,\varepsilon)\,\geq\,\frac{1}{\Vert \ell^{\varepsilon}_{\bxi}\Vert^2_{L^2}}\,,\quad\mbox{for all $\bxi\in\om$ s.t. $\B_{\infty}(\bxi,\varepsilon)\subset\om$}\,.\end{equation}
On the other hand, and still with $\varepsilon>0$ fixed, as $\B_{\infty}(\bxi,\varepsilon)\subset\om$, let 
$h_\varepsilon:=\frac{1_{\B_{\infty}(\bxi,\varepsilon)}}{\taue f}\in L^2(\mu)$. 
As $\om$ is compact, $\R[\x]$ is dense in $L^2(\mu)$ and therefore
there exists  a sequence $(h_n)_{n\in\N}\subset \R[\x]$ such that 
$\Vert h_n-h_\varepsilon\Vert_{L^2(\mu)}\to 0$ as $n$ increases.  Moreover, 
\begin{eqnarray*}
\phi^\varepsilon_{\bxi}(\vert h_n-h\vert)&=&\int_{\B_\infty(\bxi,\varepsilon)}\vert h_n-h_\varepsilon\vert\,\frac{d\x}{\taue}\,=\,
\int_{\B_\infty(\bxi,\varepsilon)}\frac{\vert h_n-h_\varepsilon\vert}{\taue\,f}\,f\,d\x\\
&=&\int_{\om}\frac{1_{\B_\infty(\bxi,\varepsilon)}\,\vert h_n-h_\varepsilon\vert}{\taue\,f}\,d\mu\\
&\leq&\frac{\sqrt{\mu(\om)}}{\taue\,f_*}\left(\int_{\om}(1_{\B_\infty(\bxi,\varepsilon)}\,(h_n-h_\varepsilon))^2\,d\mu\right)^{1/2}\\
&\leq&\frac{\sqrt{\mu(\om)}}{\taue\,f_*}\left(\int_{\om}(h_n-h_\varepsilon)^2\,d\mu\right)^{1/2}\,=\,\frac{\sqrt{\mu(\om)}}{\taue\,f_*}\Vert h_n-h_\varepsilon\Vert_{L^2(\mu)}\,,
\end{eqnarray*}
where $f_*=\min\,\{f(\x): \x\in\om\}$ ($>0$ as $f$ is continuous and $f>0$ on $\om$ compact). Therefore,
$\phi^\varepsilon_{\bxi}(h_n)\to\phi^\varepsilon_{\bxi}(h)$ as $n$ increases.
Then from
\[\tilde{\Lambda}^{\mu}_n(\bxi,\varepsilon)\,\leq\,\frac{\int h_n^2\,d\mu}{(\int h_n\,d\phi^{\varepsilon}_{\bxi})^2}\,,\quad\forall n\,,\]
and  since $\lim_{n\to\infty}\phi^{\varepsilon}_{\bxi}(h_n)=\phi^{\varepsilon}_{\bxi}(h)$, 
$\lim_{n\to\infty}\int h_n^2\,d\mu=\int h^2\,d\mu$,
we conclude that for all $\bxi\in\om$ with $\B_{\infty}(\bxi,\varepsilon)\subset\om$,
\[\limsup_{n\to\infty}\tilde{\Lambda}^{\mu}_n(\bxi,\varepsilon)\,\leq\,\lim_{n\to\infty}\frac{\int h_n^2\,d\mu}{(\int h_n\,d\phi^{\varepsilon}_{\bxi})^2}\,=\,
\frac{\int h^2\,d\mu}{(\int h\,d\phi^{\varepsilon}_{\bxi})^2}\,=\,\frac{1}{\Vert \ell^{\varepsilon}_{\bxi}\Vert^2_{L^2}}\,,\]
which, combined with \eqref{eq:inf}, yields \eqref{th1-1}. 
Finally, as $f>0$ on $\om$ and $f$ is continuous,
$\int_{\B_{\infty}(\bxi,\varepsilon)}\frac{1}{f}\frac{d\x}{\taue}=f(\bzeta_{\varepsilon})^{-1}$ for some
$\bzeta_{\varepsilon}\in\B_{\infty}(\bxi,\varepsilon)$, which yields \eqref{th1-2}. Moreover,
$f(\bzeta_{\varepsilon})\to f(\bxi)$ as $\varepsilon\downarrow0$ and so 
$f(\bzeta_{\varepsilon})\approx f(\bxi)$ for $\epsilon>0$ fixed and small.

(ii) Finally let $\bxi\in\mathrm{int}(\om)$ and let $\epsilon_n:=n^{-r}$ with $r>0$ so that as $n$ increases, 
$\B_{\infty}(\bxi,1/n^r)\subset\om$ for all $n$ sufficiently large, and $\tau_{\varepsilon_n}=\varepsilon_n^d=n^{-rd}$.
Recall that with $n$ fixed sufficiently large,  
\[\Vert\ell_{\bxi}^{\varepsilon_n}\Vert^2_{L^2}\,=\,1/(\tau_{\epsilon_n}f(\bzeta_{\epsilon_n}))\,=\,
n^{rd}/f(\bzeta_{\epsilon_n})\,,\]
for some $\bzeta_{\epsilon_n}\in\B_\infty(\bxi,1/n^r)$. Then for each $n$ fixed,
and with same arguments as in \eqref{aux-dense}, one obtains
\begin{eqnarray}
\nonumber
n^{-rd}\,f(\bzeta_{\epsilon_n})\,=\,\frac{1}{\Vert\ell_{\bxi}^{\varepsilon_n}\Vert^2_{L^2}}
\nonumber
&=&\inf_{h\in L^2(\mu)}\,\frac{\int h^2\,d\mu}{\ell_{\bxi}^{\varepsilon_n}(h)^2}\\
\nonumber
&=&\inf_{h\in L^2(\mu)}\,\{\int h^2\,d\mu:\: \ell_{\bxi}^{\varepsilon_n}(h)\,=\,1\}\\
\nonumber
&\leq&\inf_{h\in \R[\x]_n}\,\{\int h^2\,d\mu:\: \ell_{\bxi}^{\varepsilon_n}(h)\,=\,1\}\\
\label{growth}
&\leq&\tilde{\Lambda}^\mu_n(\bxi,1/n^{r})\,,
\end{eqnarray}
from which we deduce that with $r=1$, $n^d\,\tilde{\Lambda}^\mu_n(\bxi,1/n)\,\geq f(\bzeta_{\varepsilon_n})$, for every $n\in\N$,
and thus 
\[n^d\,\tilde{\Lambda}^\mu_n(\bxi,1/n)\,\geq \inf \{f(\x): \x\in\B_\infty(\bxi,1/n^r)\}\,,\quad\forall n\in\N\,.\]
If $f$ is continuous then $f(\bzeta_{\epsilon_n})\to\bxi$ as $n\to\infty$, and
we obtain \eqref{th1-33}. Finally,  $n^d\,\tilde{\Lambda}^\mu_n(\bxi,1/n^r)\,\geq n^{d(1-r)}f(\bzeta_{\varepsilon_n})$, for every $n\in\N$, which
yields \eqref{th1-4}.
\end{proof}
So Theorem \ref{th1} states that for fixed $\varepsilon>0$, asymptotically when $n$ increases,
$\tilde{\Lambda}^\mu_n(\bxi,\varepsilon)^{-1}$ converges to the $L^2(\mu)$-norm
of the linear functional $\ell^\varepsilon_{\bxi}\in L^2(\mu)$ which can be viewed
as an $\varepsilon$-average (or $\varepsilon$-regularization) of the point evaluation 
$\delta_{\bxi}\not\in L^2(\mu)$. However,
and as expected, $\Vert\ell^\varepsilon_{\bxi}\Vert_{L^2}$ increases with $1/\varepsilon$. 

On the other hand, for the standard Christoffel function,
$\Lambda^\mu_n(\bxi)^{-1}$ is the $L^2(\mu)$-norm of the the point evaluation
linear functional $\delta_{\bxi}$ viewed as an element of the finite dimensional subspace
$(\R[\x]_n,\langle\cdot,\cdot\rangle_{L^2})$ of $L^2(\mu)$. But then of course this norm
is unbounded as $n$ increases because $\delta_{\bxi}\not\in L^2(\mu)$.

\begin{rem}
\label{rem-regularization}
In Theorem \ref{th1}(i) one restricts to $\bxi\in\mathrm{int}(\om)$ 
such that $\B_\infty(\bxi,\varepsilon)\subset\om$. The reason is that 
if $\bxi\in\mathrm{int}(\om)$ but $\mu(\B_\infty(\bxi,\varepsilon)\cap(\R^d\setminus\om))>0$, then
for all $h\in L^2(\mu)$,
\[\int_{\B_\infty(\bxi,\varepsilon)}h\,\frac{d\x}{\varepsilon^d}\,=\,
\int_{\om} h\,\frac{1_{\B_\infty(\bxi,\varepsilon)\cap\om}}{\varepsilon^df}\,d\mu\,=\,
\langle h\,,\,\frac{1_{\B_\infty(\bxi,\varepsilon)\cap\om}}{\varepsilon^df}\rangle\,,\]
and therefore, \eqref{th1-2} becomes
\begin{equation}
\label{lim-outside}
\lim_{n\to\infty}\varepsilon^d\,\tilde{\Lambda}^{\mu}_n(\bxi,\varepsilon)^{-1}
\,=\,\int_{\B_{\infty}(\bxi,\varepsilon)\cap\om}\frac{1}{f}\,\frac{d\x}{\varepsilon^d}\,,\end{equation}
while if $f$ is continuous then \eqref{th1-3} becomes
\[\lim_{n\to\infty}\varepsilon^{-d}\,\tilde{\Lambda}^{\mu}_n(\bxi,\varepsilon)\,=\,f(\bzeta_{\varepsilon})\,\frac{\taue}{\mathrm{vol}(\B_\infty(\bxi,\varepsilon)\cap\om)}\,,\]
for some $\bzeta_\varepsilon\in\B_\infty(\bxi,\varepsilon)\cap\om$.
So now the limit in \eqref{lim-outside} is the product of 

- the nice term $f(\bzeta_\varepsilon)$ related to the density
of $f$ in $\B_\infty(\bxi,\varepsilon)$, with

- the multiplicative term
$\varepsilon^d/\mathrm{vol}(\B_\infty(\bxi,\varepsilon)\cap\om)$,
which may be hard to compute for arbitrary $\om$. 

On the other hand, the limit in \eqref{lim-outside} still converges to $f(\bxi)$ when $\varepsilon\downarrow 0$.

\end{rem}

\subsection*{Comparing with asymptotics of  the standard $\Lambda^\mu_n$.}
Recall that $s(n)={d+n\choose d}=O(n^d)$.
Notice that under some restrictive conditions, if $\bxi\in\mathrm{int}(\om)$,
\begin{equation}
\label{unknown}
\lim_{n\to\infty}s(n)\,\Lambda^\mu_n(\bxi)\,=\,\frac{f(\bxi)}{\omega_E(\x)}\,,\end{equation}
where $\omega_E(\x)$ is the (in general unknown) density of the equilibrium measure of $\om$. 
See for instance some examples in \cite[\S 9.7]{dunkl-xu}, \cite{Kroo-a,Kroo,Xu} and  \cite[\S 4.2]{lass-book} 
.\\

Then Eq. \eqref{unknown}  is to be compared with \eqref{th1-3}, i.e.,
\[\lim_{n\to\infty}\varepsilon^{-d}\,\tilde{\Lambda}^\mu_n(\bxi,\varepsilon)
\,=\,f(\bzeta_\varepsilon)\,,\quad\mbox{for some $\bzeta_{\varepsilon}\in\B_\infty(\bxi,\varepsilon)$,}\]
when $\varepsilon>0$ is fixed, $\B_{\infty}(\bxi,\varepsilon)\subset\om$, and with
\eqref{th1-33}-\eqref{th1-4}, i.e.:
\[\liminf_{n\to\infty}n^d\,\tilde{\Lambda}^\mu_n(\bxi,1/n)\,\geq\,f(\bxi)\,;\quad
\lim_{n\to\infty}n^d\,\tilde{\Lambda}^\mu_n(\bxi,1/n^r)\,=\,+\infty\quad\mbox{if $r<1$}\,,\]
when $\varepsilon =1/n$ or $\varepsilon =1/n^r$ 
is not fixed anymore.

So \eqref{unknown} is more precise but requires strong assumptions on $f$, and 
is not very helpful for practical computation since the density 
$\omega_E$ of the equilibrium measure of $\om$ is not known in general. 

On the other hand,
\eqref{th1-3} is obtained with a rather weak assumption on $f$, and  provides an information of the density $f$ in the box $\B_\infty(\bxi,\varepsilon)$, as well as an approximation of $f(\bxi)$ when $\varepsilon$ is small. However it holds for points
$\bxi\in\mathrm{int}(\om)$ whose $\Vert\cdot\Vert_\infty$-distance  
to the boundary $\partial\om$ is at least $\varepsilon/2>0$.

Notice that as indicated in \eqref{th1-4}, varying $\varepsilon$ with $n$ as $\varepsilon=1/n^r$, with $r<1$,
is not a good idea if one would like to compare the limit
of $n^d\,\tilde{\Lambda}^\mu_n(\bxi,1/n^r)$ with that of $n^d\,\Lambda^\mu_n(\bxi)$ in \eqref{unknown}.
To obtain a meaningful  bounded limit 
would require at least $\varepsilon=1/n$.

 \subsection*{Support inference}
 
 As already mentioned in the introduction, one striking property of the Christoffel function $\Lambda^\mu_n$ is the dichotomy of its 
 asymptotic behavior with $n$, depending on whether $\bxi\in\mathrm{int}(\om)$
 or $\bxi\not\in\om$. As we next show, 
 $\tilde{\Lambda}^\mu_n$ inherits the same highly desirable property. \\
 
 \noindent
 {\bf 1. Inside $\om$.} 
 First with $\varepsilon>0$ fixed, by Theorem \ref{th1}(i), 
 for every $\bxi$ such that $\B_{\infty}(\bxi,\varepsilon)\subset\om$,
 \[\lim_{n\to\infty}\varepsilon^{-d}\,\tilde{\Lambda}^\mu_n(\bxi,\varepsilon)^{-1}\,\to\,\int_{\B_{\infty}(\bxi,\varepsilon)}(1/f)\,\frac{d\x}{\taue}\,,\]
  so that the sequence $(\tilde{\Lambda}^\mu_n(\bxi,\varepsilon))_{n\in\N}$ is 
  bounded (as it converges). So clearly, $(\tilde{\Lambda}^\mu_n(\bxi,\varepsilon))^{-1}$ is bounded, uniformly in $n$.
  Next, when $\varepsilon$ is allowed to vary with $n$, like  $\varepsilon=1/n^r$ with $r>0$, then from
 \eqref{growth} in the proof of Theorem \ref{th1},
 \[\tilde{\Lambda}^\mu_n(\bxi,1/n^r)^{-1}\,\leq\,\frac{n^{rd}}{\gamma}\,,\quad\forall\bxi\in\mathrm{int}(\om)\,,\]
 that is, the growth of $\tilde{\Lambda}^\mu_n(\bxi,1/n^r)^{-1}$ is at most polynomial in
 $n$.\\

\noindent
{\bf 2. Outside $\om$.}

\begin{lem}
 \label{lem-aux}
 Let $\bxi\not\in\om$ and fix $\varepsilon>0$.
  Then the growth of  $\tilde{\Lambda}^\mu_n(\bxi,\varepsilon)^{-1}$ is at least exponential in $n$.
  Similarly, the growth of  $\tilde{\Lambda}^\mu_{2n}(\bxi,1/n)^{-1}$ is at least exponential in $n$.
  \end{lem}
A proof is postponed to \S \ref{sec:proof}.\\

So, similarly as for the standard Christoffel function
$\Lambda^\mu_n(\bxi)$, we also obtain for $\tilde{\Lambda}^\mu_n(\bxi,\varepsilon)$ a dichotomy in 
its asymptotic behavior as $n$ increases, depending on whether
$\bxi\not\in\om$ (exponential decrease to zero) or $\bxi\in\om$ and at $\Vert\cdot\Vert_\infty$-distance
$\varepsilon$ of the boundary $\partial\om$ (convergence to a nonzero value).  
A similar result holds for $\tilde{\Lambda}^\mu_{2n}(\bxi,1/n)$ when $\epsilon>0$ is allowed to vary with $n$, like e.g., $\varepsilon=1/n$.

\section{Conclusion}
 We have introduced an extended version $\tilde{\Lambda}^\mu_n(\x,\varepsilon)$ of the standard Christoffel function 
 $\Lambda^\mu_n(\x)$ for a measure  $\mu$ on a compact set $\om\subset\R^d$. One main motivation was
 to improve the asymptotic behavior of $\Lambda^\mu_n(\x)$ as $n$ increases and $\x\in\mathrm{int}(\om)$,
 from a computational point of view.  Indeed, when $s(n)\Lambda^\mu_n(\x)$ converges (under certain assumptions), 
 the limit is of the form $f(\x)/\omega_E(\x)$ where $f$ (resp. $\omega_E$) is the density of $\mu$ 
 (resp. of the equilibrium measure of $\om$) w.r.t. Lebesgue measure. As $\omega_E$ is in general unknown,
 this result is of limited value from a computational viewpoint. In contrast, if $f$ is continuous,
 then with $\varepsilon>0$ fixed and as $n$ grows,
  $\varepsilon^{-d}\tilde{\Lambda}^\mu_n(\x,\varepsilon)\to f(\bzeta_\varepsilon)$ for some 
  $\bzeta_\varepsilon\in\B_\infty(\x,\varepsilon)$.
  So if $\varepsilon$ is small, one obtains an approximation of $f(\x)$ which converges to $f(\x)$ if $\varepsilon\downarrow 0$.
 In addition, the extended Christoffel function still inherits highly desirable properties of $1/\Lambda^\mu_n$,
 namely (i) the dichotomy of its symptotic behavior (at most polynomial growth versus 
 at least exponential growth) depending on whether $\x$ is inside or outside $\om$, and (ii) an efficient
 computation of its closed form expression as a polynomial of $(\x,\varepsilon)$.
 
 \section{Appendix}
 
 \subsection{Affine invariance of the Christoffel function}
 
 Let $\om\subset\R^d$ be compact and let $\bxi\in\R^d$ be fixed. With $r>0$, consider
 the linear transformation $T:\R^d\to\R^d$ defined by
 \begin{equation}
 \label{def-T}
 \x\,\mapsto\,T(\x)\,:=\,(\x-\bxi)/r\,,\quad \x\in\R^d\,.\end{equation}
 If $\mu$ is a Borel measure on $\om$, let $\nu$ be the pushforward measure of
 $\mu$ by the mapping $T$, that is,
 \[\nu(C)\,:=\,\mu(T^{-1}(C))\,,\quad\forall C\in\mathcal{B}(\om)\,.\]
 The support of $\nu$ is the set $T(\om)$, and
 \[\int_{T(\om)} p(\y)\,d\nu(\y)\,=\,\int_{\om}p(T(\x))\,d\mu(\x)\,,\quad\forall p\in\R[\x]\,.\]
 \begin{lem}
 \label{lem:affine}
 Let $T$ be as in \eqref{def-T} and with $\mu$ supported on $\om$, let $\nu$ be its 
 pushforward by $T$ (hence supported on $T(\om)$). Then
 \begin{eqnarray}
 \label{affine-1}
 \Lambda^\nu_n(T(\x))&=&\Lambda^\mu_n(\x)\,,\quad\forall\x\in\R^d\\
 \label{affine-2}
 \tilde{\Lambda}^\nu_n(0,\varepsilon/r)&=&\tilde{\Lambda}^\mu_n(\bxi,\varepsilon)\,,\quad\forall\bxi\in\R^d\,.
 \end{eqnarray}
 \end{lem}
 \begin{proof}
 \begin{eqnarray*}
 \Lambda^\nu_n(T(\x))&=&\min_{p\in\R[\y]_n}\,\{\,\int_{T(\om)}p(\y)^2\,d\nu(\y):\:p(T(\x))\,=\,1\,\}\\
 &=&\min_{p\in\R[\y]_n}\,\{\,\int_{\om}p(T(\y))^2\,d\mu(\y):\:p(T(\x))\,=\,1\,\}\,=\,\Lambda^\mu_n(\x)\,,
  \end{eqnarray*}
 where we have used that $\{p\circ T:p\in\R[\x]_n\}$ generates $\R[x]_n$.
Moreover, observe that
\[\int_{\B(0,\varepsilon/r)}p(\u)\,\frac{d\u}{(\varepsilon/r)^d}\,=\,
\int_{\B_\infty(T(\bxi),\varepsilon)}p(T(\x))\,\frac{d\x}{\varepsilon^d}\,,\]
and since $T(\bxi)=0$, we also obtain
\begin{eqnarray*}
 \tilde{\Lambda}^\nu_n(0,\varepsilon/r)&=&\min_{p\in\R[\y]_n}\,\{\,\int_{T(\om)}p(\y)^2\,d\nu(\y):\:\int_{\B_{\infty}(0,\varepsilon/r)}
 p(\u)\,\frac{d\u}{(\epsilon/r)^d}=\,1\,\}\\
 &=&\min_{p\in\R[\y]_n}\,\{\,\int_{\om}p^2(T(\x))\,d\mu(\x):\:\int_{\B_{\infty}(\bxi,\varepsilon)}p(T(\v))\,\frac{d\v}{\varepsilon^d}=\,1\,\}\\
 &=&\tilde{\Lambda}^\mu_n(\bxi,\varepsilon)\,.
 \end{eqnarray*}
 \end{proof}
 \subsection{Proof of Lemma \ref{lem-aux}}
 \label{sec:proof}
  \begin{proof}
   Let $T:\R^d\to\R^d$ be as in \eqref{def-T} with
 $r>0$ such that $T(\om)\subset\B_2(0,1)$ (where
 $\B_2(\bxi,\tau):=\{\x: \Vert \x-\bxi\Vert_2<\tau\}$),
  and let the measure $\nu$ on $T(\om)$ 
 be the pushforward of $\mu$ by $T$. 
 Observe that $T(\bxi)=0\not\in T(\om)$ since $\bxi\not\in\om$.
 By the Bernstein-Markov property
 of Lebesgue measure on the box $\B_{\infty}(0,\varepsilon/r)$, for every $\eta>0$ there exists a constant $C_\eta>0$, such that
 for all $p\in\R[\x]$,
  \[\sup\{\vert p(\x)\vert:\x\in\B_{\infty}(0,\varepsilon/r)\}\,\leq\,C_\eta\,(1+\eta)^{\mathrm{deg}(p)}\,\int_{\B_{\infty}(0,\varepsilon)}p(\x)^2\frac{d\x}{\tauer}\,,\]
 see \cite[p. 51]{lass-book}. Therefore if $p(0)=1$ then,
 \[\int_{\B_{\infty}(0,\varepsilon/r)}p(\x)^2\frac{d\x}{\tauer}\,\geq\,C_\eta^{-1}\,(1+\eta)^{-\mathrm{deg}(p)}\,.\]
Next, for every $\delta\in (0,1)$,
 the \emph{needle} polynomial $q\in\R[t^2]_{n(\delta)}$  introduced in Kr\'oo and Lubinsky \cite{Kroo},
satisfies:
 \begin{equation}
\label{aux-60}
 q(0)=1\,;\quad \vert q(\Vert\x\Vert_2)\vert\,\leq\,2\cdot 2^{-\delta\,n(\delta)}\,
  \quad\forall \x\in \B_2(0,1)\setminus\B_2(0,\delta)\,,
  \end{equation}
  and in particular,
  \[q(\Vert\x\Vert_2)^2\,\leq\,4\cdot 2^{-\delta\,2n(\delta)}\,,\quad\forall \x\in T(\om)\,,\]
  whenever $\B_2(0,\delta)\cap T(\om)=\emptyset$ (which happens whenever $\delta$ is sufficiency small
  as $0\not\in T(\om)$).
  So with $q$ as in \eqref{aux-60}, letting 
 $v=q^2\in\R[\x]_{2n(\delta)}$,  one obtains $v(0)=1$ and
 \[\int_{\B_{\infty}(0,\varepsilon/r)}v(\x)\frac{d\x}{\tauer}\,\geq\,C_\eta^{-1}\,(1+\eta)^{-n(\delta)}\,.\]
 Therefore, we deduce that
  \[\tilde{\Lambda}^\nu_{2n(\delta)}(0,\varepsilon/r)\,\leq\,\frac{\int_{T(\om)} v^2\,d\nu}{\int_{\B_\infty(0,\varepsilon/r)} v\,
  d\x/{\tauer}}\,
  \leq\,C_\eta\,(1+\eta)^{n(\delta)}\,  \int_{T(\om)} v^2\,d\nu\,,\]
  whenever $\B_2(0,\delta)\cap T(\om)=\emptyset$, and therefore 
  \[\tilde{\Lambda}^\nu_{2n(\delta)}(0,\varepsilon/r)\,\leq\,16\,C_\eta\,(1+\eta)^{n(\delta)}\,2^{-\delta\,4n(\delta)}\,.\]
  So recalling that $\bxi\not\in\om$ (hence $0\not\in T(\om)$), fix $\delta\in (0,1)$ such that 
  $\B_2(0,\delta)\cap T(\om)=\emptyset$.
    Then by choosing $\eta$ such that $a:=(1+\eta)/2^{4\delta}<1$, e.g. $\eta<4\delta\,\ln{2}$, one obtains:
    \[\tilde{\Lambda}^\nu_{2n(\delta)}(0,\varepsilon/r)\,\leq\,16\,C_\eta\,a^{n(\delta)}\,,\]
 which combined with Lemma \ref{lem:affine} yields
    \[\tilde{\Lambda}^\mu_{2n(\delta)}(\bxi,\varepsilon)\,\leq\,16\,C_\eta\,(\sqrt{a})^{2n(\delta)}\,,\]
 which in turn shows that $\tilde{\Lambda}^\mu_{2n}(\bxi,\varepsilon)$ decreases exponentially  fast to zero with $n$, 
 whenever  $\bxi\not\in\om$.
 
 Next, if $\varepsilon=1/n$ for every $n$, then agin let $\delta>0$ be such that $\B_2(0,\delta)\cap T(\om)=\emptyset$.
As we did above with $\varepsilon$ fixed, 
 \[\tilde{\Lambda}^\nu_{2n(\delta)}(0,1/n(\delta)r)\,\leq\,16\,C_\eta\,(1+\eta)^{n(\delta)}\,2^{-\delta\,4n(\delta)}\,.\]
 Again, by choosing $\eta$ such that $a:=(1+\eta)/2^{4\delta}<1$, e.g. $\eta<4\delta\,\ln{2}$, one obtains:
    \[\tilde{\Lambda}^\nu_{2n(\delta)}(0,1/n(\delta)r)\,\leq\,16\,C_\eta\,a^{n(\delta)}\,,\]
 which combined with Lemma \ref{lem:affine} yields
    \[\tilde{\Lambda}^\mu_{2n(\delta)}(\bxi,1/n(\delta))\,\leq\,16\,C_\eta\,(\sqrt{a})^{2n(\delta)}\,,\]
 which shows that $\tilde{\Lambda}^\mu_{2n}(\bxi,1/n)$ decrease exponentially  fast to zero whenever $\bxi\not\in\om$.
  \end{proof}


\begin{thebibliography}{las}
\bibitem{dunkl-xu}
C.~F. Dunkl, Yuan Xu. \emph{Orthogonal Polynomials of Several Variables}, 2nd Edition, Cambridge University Press, Cambridge, UK, 2014.
\bibitem{Kroo-a}
A. Kro\'{o}, D.~S. Lubinsky. Christofffel functions and universality on the boundary of the ball,
\emph{Acta. Math. Hungarica} {\bf 140}, pp. 117--133, 2013.
\bibitem{Kroo}
A. Kro\'{o}, D.~S. Lubinsky. Christofffel functions and universality in the bulk for multivariate polynomials,
\emph{Canad. J. Math.} {\bf 65}(3), pp. 600--620, 2013.
\bibitem{nips}
J.B. Lasserre, E. Pauwels. Sorting out typicality via the inverse moment matrix SOS polynomial,
in \emph{Advances in Neural Information Processing Systems}, D.D. Lee, M. Sugiyama,
U.V. Luxburg, I. Guyon and R. Garnett (Eds.), Curran Associates, Inc., pp. 190--198, 2016.
\bibitem{cras-classif}
J.~B. Lasserre. On the Christoffel function and classification in data analysis,
\emph{Comptes Rendus Math\'ematique} {\bf 360}, pp. 919--928, 2022.
\bibitem{cras-disintegration}
J.~B. Lasserre. A disintegration of the Christoffel function, \emph{Comptes Rendus Math\'ematique} {\bf 360}, pp. 1071--1079, 2022.
\bibitem{pell}
J.~B. Lasserre. Pell's equation, sum-of-squares, and equilibrium measures of compacts sets.
To appear in \emph{Comptes Rendus Math\'ematique}, 2023.{\tt arXiv:2210.07608}
\bibitem{lass-book}
J.~B. Lasserre, E. Pauwels, M. Putinar. \emph{The Christoffel-darboux Kernel for Data Analysis},
Cambridge Monographs on Applied and Computational Mathematics, Cambridge University Press,
Cambridge, UK, 2022.
\bibitem{advc}
J.~B. Lasserre, E. Pauwels.
The empirical Christoffel  function with applications in data analysis,
\emph{Adv. Comput. Math.} {\bf 45}, pp. 1439--1468, 2019.
\bibitem{Nevai}
P. Nevai. G\'eza Freud, orthogonal polynomials and Christoffel functions. A case study. 
\emph{J. Approx. Theory} {\bf 48}(1), pp. 3--167, 1986.
\bibitem{Simon}
B. Simon. The Christoffel Darboux kernel. In \emph{Perspectives in Partial Differential Equations,  Harmonic Analysis and 
Applications}, Proc. Symp. Pure Math., vol 79, American Mathematical Society, Providence, RI, 2008, pp. 295--335. 
\bibitem{Totik}
H. Stahl, V. Totik. \emph{General Orthogonal Polynomials}, Encyclopedia of Mathematics and its Applications, 
Cambridge University Press, Cambridge, 1992.
\bibitem{Xu}
Y. Xu. Asymptotics of the Christoffel functions on a simplex in $\R^d$, \emph{J. Approx. Theory} {\bf 99} (1), pp. 122--133, 1999.
\end{thebibliography}
\end{document}